\documentclass[reqno]{article}
\usepackage{amsthm}
\usepackage{amsfonts}
\usepackage{amsmath}
\usepackage{mathrsfs}
\theoremstyle{definition}
\newtheorem{theorem}{Theorem}
\newtheorem{lemma}[theorem]{Lemma}

\newtheorem{proposition}[theorem]{Proposition}
\newtheorem{corollary}[theorem]{Corollary}
\newtheorem{definition}[theorem]{Definition}
\newtheorem{notation}[theorem]{Notation}

\def\N{\mathbb{N}}

\def\myslash{\mbox{Slash}(\Sigma)}

\begin{document}

\title{A Dichotomy in Machine Knowledge}
\author{Samuel A.~Alexander\thanks{Email:
alexander@math.ohio-state.edu}\\
\emph{Department of Mathematics, the Ohio State University}}

\maketitle

\begin{abstract}
We show that a machine, which knows basic logic and arithmetic and basic axioms of knowledge,
and which is factive (knows nothing false), can either know that it is factive,
or know its own G\"{o}del number, but not both.
\end{abstract}

\section{Introduction}

This is not a paper about artificial intelligence or conscious machines.
But for motivational purposes,
temporarily imagine we did have such a machine, whatever exactly that entails.
We could ask this machine to tell us everything it knows: to enumerate its knowledge.
It would then begin telling us various things: ``$1+1=2$'', ``there are
infinitely many primes'', and so on.
It would also tell us things about its knowledge: ``I know that $1+1=2$''.
We can further restrict our request, asking the machine to list only
facts which it knows in the language of (say) Peano Arithmetic extended by
a connective $K$ for knowledge (formalized in Section~\ref{baselogicsection}):
``I know $1+1=2$'' becomes ``$K(1+1=2)$''.
We can at least say one thing about this list of knowledge:
it is recursively enumerable (at least if we let the machine enumerate
without outside disturbance).

Conscious or intelligent computers are beyond the scope of this paper,
but the above shows how we can study \emph{machine knowledge}
anyway.  Namely, we study recursively enumerable sets of formulas in a language
which includes a modal operator for knowledge.  In this paper, the language will
be that of $PA$ augmented by a unary knowledge operator $K$.
We abuse language and identify a machine with its set of knowledge.
For example, when we say a machine is ``factive'', we mean that its set of
of known formulas are all true (in a background model).  If we say a machine ``knows Peano arithmetic'',
we mean that the axioms of PA
are in the r.e.~set.

\emph{Trivial Examples.}  There is the know-nothing machine: we ask it to enumerate
its knowledge and it lists nothing.  This machine is vacuously factive.  We can say it satisfies
the schema $K(\phi)\rightarrow\phi$.  The hypothesis, $K(\phi)$, means that $\phi$ is among the
list of known formulas.  The conclusion $\phi$ means $\phi$ is really true.  Again, there is
the all-knowing machine: it lists every formula.  This machine is not factive.
It does not satisfy $K(0\not=0)\rightarrow (0\not=0)$ (the hypothesis is true, the conclusion false).
Again, there is the machine which knows exactly the consequences of $PA$.
This machine is factive (if the background universe is $\N$) but it does
not know itself to be factive (the schema $K(K(\phi)\rightarrow\phi)$ fails).

We are primarely concerned with machines which have the following properties.
(which we call the \emph{axioms of a knowing machine}).
\begin{itemize}
\item Knowledge of Tautology: $K(\phi)$ whenever $\phi$ is tautological.
\item Knowledge Modus Ponens: $K(\phi\rightarrow\psi)\rightarrow K(\phi)\rightarrow K(\psi)$.
\item Knowledge of Arithmetic: $K(\phi)$ whenever $\phi$ is an axiom of $PA$.
\item Closure: $K(\phi)$ whenever $\phi$ is one of the above assumptions.
\item Factivity: $K(\phi)\rightarrow\phi$.  Everything known is true.
\end{itemize}

Each of these assumptions is plausible, requiring little of the machine in question.
All are standard in epistemology.  We are also interested in two additional properties, and the goal
is to show that, together with the above basics, these properties are individually possible (for some $e\in\N$) but are
mutually inconsistent:
\begin{itemize}
\item Knowledge of Factivity: $K(K(\phi)\rightarrow \phi)$.
\item Knowledge of having G\"{o}del number $e$:  $K(K(\phi)\leftrightarrow \ulcorner \phi\urcorner \in W_e)$,
where ``$\ulcorner \phi\urcorner\in W_e$'' abbreviates a canonical sentence expressing that
the G\"{o}del number of $\phi$ in the $e$th r.e.~set.
\end{itemize}

In \cite{reinhardt}, it was shown that a knowing machine cannot know its own G\"{o}del number.
However, this ``implicitly'' assumed Knowledge of Factivity.
I say this requirement was ``implicit'' because it was not explicitly
spelled out, but rather part of a closure requirement:
in essence, in the above list of assumptions, ``Closure'' and ``Factivity'' were permuted.

\section{A Very Simple Modal Logic}
\label{baselogicsection}

In this section we formalize quantified modal logic.  There are many ways to do this.
We take a very simple and weak approach.
This approach is so weak that (unlike many treatments of modal logic)
it does not actually depart from standard first-order logic.
By taking such a weak approach to modal logic, we eliminate a lot of
technical difficulty.

\begin{definition}
Suppose $\mathscr{L}$ is a first-order language and $K$ is a symbol not in $\mathscr{L}$ (we will call $K$
an \emph{unary modal operator}).
The first-order language $\mathscr{L}(K) $\emph{obtained by weakly extending $\mathscr{L}$ by $K$} is defined inductively as follows:
\begin{enumerate}
\item All the function, constant, and predicate symbols of $\mathscr{L}$ are also in $\mathscr{L}(K)$.
\item For every formula $\phi$ of $\mathscr{L}(K)$,
$\mathscr{L}(K)$ contains a new $0$-ary predicate symbol $K_{\phi}$.
\end{enumerate}
\end{definition}

\begin{notation}
\label{notation2}
Write $K(\phi)$ for $K_{\phi}$.
The abbreviation $K(\phi)$
may be pronounced ``$\phi$ is known'' or ``I know $\phi$'' or (for brevity) ``know $\phi$''.
\end{notation}

The notation is applied inductively.  For example, we may write $K(K(\phi))$
to abbreviate $K_{K_{\phi}}$.

\emph{Warning.}  Notation~\ref{notation2} does not play nicely with variable substitution.
For instance $K(x=y)(x|z)$ (the result of substituting $z$ for $x$)
is $K(x=y)$, not $K(z=y)$.
Failure to heed this warning can lead to philosophical paradoxes.
In our treatment, the schemas $(\forall x\phi)\rightarrow \phi(x|t)$ and
$(t_1=t_2)\rightarrow \phi(x|t_1)\rightarrow \phi(x|t_2)$ are valid, and the Substitution
Theorem holds: after all, we have not left classical first-order logic!  These things
can fail in some treatments of modal logic (see Shapiro \cite{shapiro}).

Hereafter, let $\mathscr{L}(K)$ be the language of $PA$ weakly extended by $K$.

\begin{definition}
\label{machinedef}
By a \emph{knowing entity} we mean an $\mathscr{L}(K)$-structure $\mathscr{M}$
with universe $\mathbb{N}$,
interpreting symbols of $PA$ as usual,
such that $\mathscr{M}$ satisfies the axioms of a knowing machine from the Introduction.
By a \emph{knowing machine} we mean a knowing entity $\mathscr{M}$
with the property that $\{\ulcorner\phi\urcorner\,:\,\mathscr{M}\models K(\phi)\}$ is r.e.
\end{definition}

Note that an entity (hence a machine) is completely determined by the formulas which it knows.
In fact, it would be possible to reformulate Definition~\ref{machinedef} and define entities and machines
to be sets of formulas; this is what Carlson does \cite{carlson2000} (pp.~59, 61).  For our purposes, Definition~\ref{machinedef} is
more convenient.

\section{An Inconsistency Result}
\label{section3}

William Reinhardt \cite{reinhardt} showed that a knowing machine cannot simultaneously
know that it is factive, and also know its own G\"{o}del number (although everything was
formalized in different ways).  In this section we offer a streamlined
proof of this result.

\begin{proposition}
\label{prop5}
Let $e\in\N$ and
let $\Sigma$ be the set of axioms of a knowing machine together with Knowledge of Factivity and Knowledge of having G\"{o}del number $e$.
Then $\Sigma$ is inconsistent.
\end{proposition}

\begin{proof}
By G\"{o}del's diagonal lemma, there is a sentence $\phi$ such that Peano Arithmetic
proves $\phi\leftrightarrow \ulcorner \phi\urcorner\not\in W_e$.

Work in $\Sigma_0=PA\cup\{K(\phi)\rightarrow\phi,K(\phi)\leftrightarrow \ulcorner\phi\urcorner\in W_e\}$ (a subset of the
consequences of $\Sigma$).
By PA, we have 
$\phi\leftrightarrow\ulcorner\phi\urcorner\not\in W_e$.
Combining this with $K(\phi)\leftrightarrow \ulcorner\phi\urcorner\in W_e$, we have $K(\phi)\leftrightarrow\neg\phi$.
Assuming $\neg\phi$, we obtain $K(\phi)$, and then by $K(\phi)\rightarrow\phi$, we obtain $\phi$.
Altogether, this establishes $\phi$.

I proved $\phi$ from $\Sigma_0$.
Thus there are finitely many axioms $\sigma_1,\ldots,\sigma_n$ from $\Sigma_0$
such that $\sigma_1\rightarrow\cdots\rightarrow\sigma_n\rightarrow\phi$ is a tautology.
Now
\begin{align*}
\Sigma &\models K(\sigma_1\rightarrow \cdots \rightarrow \sigma_n\rightarrow \phi) & \mbox{(Knowledge of Tautology)}\\
\Sigma &\models K(\sigma_1)\rightarrow \cdots \rightarrow K(\sigma_n)\rightarrow K(\phi) & \mbox{(Knowledge Modus Ponens)}\\
\Sigma &\models K(\sigma_1)\wedge \cdots \wedge K(\sigma_n) &\mbox{($*$)}\\
\Sigma &\models K(\phi) &\mbox{(Modus Ponens)}\\
\Sigma &\models \ulcorner\phi\urcorner\in W_e. &\mbox{(Knowledge of having code $e$)}
\end{align*}
Line $(*)$ is true because for every element $\sigma_i$ of $\Sigma_0$, $K(\sigma_i)$ is an axiom in $\Sigma$; this is the only place
where we invoke Knowledge of Factivity (one of the $\sigma_i$ being an instance of Factivity).
So $\Sigma\models \phi$, $\Sigma\models \ulcorner\phi\urcorner\in W_e$, and $\Sigma\models \phi\leftrightarrow \ulcorner\phi\urcorner\not\in W_e$,
establishing inconsistency.
\end{proof}

\begin{theorem}
\label{theorem5}
\emph{(Reinhardt)}
There is no knowing machine satisfying Knowledge of Factivity and Knowledge of having G\"{o}del number $e$,
for any $e\in\N$.
\end{theorem}

\begin{proof}
Such a machine would satisfy the set $\Sigma$ from Proposition~\ref{prop5}.
\end{proof}

\section{Consistency of a machine knowing itself to be factive}

In this section, we exhibit a knowing machine which possesses Knowledge of Factivity.
We attempt to streamline
Appendix B in Timothy Carlson's paper \cite{carlson2000}.  Carlson himself streamlined
an argument due to Shapiro \cite{shapiro}.  In both,
Kleene's \cite{kleene} \emph{slash} operator was used; we short-circuit it.

\begin{definition}
Let $\Sigma$ be the set of axioms of a knowing machine together with the axioms of Peano Arithmetic.
Let $\myslash$ be the $\mathscr{L}(K)$-model which has universe $\N$,
interprets symbols of $PA$ in the usual way, and interprets knowledge inductively as follows:
\[
\mbox{$\myslash\models K(\phi)$ iff $\myslash\models \phi$ and $\Sigma\models\phi$.}\]
\end{definition}

\begin{lemma}
\label{slashlemma}
$\myslash\models \Sigma$.  Also, $\myslash$ satisfies the schema $K(K(\phi)\rightarrow\phi)$.
\end{lemma}

\begin{proof}\hfill
\begin{itemize}
\item \emph{(Knowledge of Tautology)} If $\phi$ is a tautology, then $\myslash\models\phi$
and $\Sigma\models\phi$, so $\myslash\models K(\phi)$.
\item \emph{(Knowledge Modus Ponens)} Suppose that $\myslash\models K(\phi\rightarrow\psi)$ and $\myslash\models K(\phi)$.
This means $\myslash\models \phi\rightarrow\psi$, $\Sigma\models\phi\rightarrow\psi$,
$\myslash\models \phi$, and $\Sigma\models\phi$.  By Modus Ponens, $\myslash\models \psi$
and $\Sigma\models\psi$.  So $\myslash\models K(\psi)$.
\item \emph{(Knowledge of Arithmetic)}
If $\phi$ is an axiom of $PA$, then $\myslash\models\phi$
since $\myslash$ has universe $\N$ and interprets symbols of $PA$ in the usual way.
Also, $\Sigma\models \phi$ since $\Sigma$ contains $K(\phi)$ (Knowledge of Arithmetic)
as well as $K(\phi)\rightarrow\phi$ (Factivity).  Altogether, $\myslash\models K(\phi)$.
\item \emph{(Closure)}
If $\phi$ is an instance of Knowledge of Tautology, Logic, or Arithmetic,
then $\myslash\models\phi$ by the above items.  And certainly $\Sigma\models\phi$.
So $\myslash\models K(\phi)$.
\item \emph{(Factivity)}
If $\myslash\models K(\phi)$, then by definition $\myslash\models \phi$.
\item
\emph{(Knowledge of Factivity)} Since $\Sigma$ contains Factivity as an axiom, $\Sigma\models K(\phi)\rightarrow\phi$.
We showed $\myslash\models K(\phi)\rightarrow\phi$.
Together these show $\myslash\models K(K(\phi)\rightarrow\phi)$.
\end{itemize}
\end{proof}

\begin{corollary}
For any $\phi$, $\myslash\models K(\phi)$ iff $\Sigma\models\phi$.
\end{corollary}

\begin{proof}
($\Rightarrow$) By definition.  ($\Leftarrow$) Suppose $\Sigma\models\phi$.
By Lemma~\ref{slashlemma}, $\myslash\models\Sigma$.  Thus $\myslash\models\phi$.
Together, this shows $\myslash\models K(\phi)$.
\end{proof}

\begin{theorem}
There is a knowing machine which possesses Knowledge of Factivity.
\end{theorem}

\begin{proof}
One such knowing machine is $\myslash$.
It is a knowing entity by Lemma~\ref{slashlemma}.
It is a machine because
$\{\ulcorner\phi\urcorner\,:\,\myslash\models K(\phi)\} = \{\ulcorner\phi\urcorner\,:\,\Sigma\models \phi\}$
is r.e.
\end{proof}

\section{Consistency of knowing one's own G\"{o}del number}

In this section we will construct a knowing machine which knows its own G\"{o}del number.
By Section~\ref{section3}, we cannot hope for such a machine to also know itself to be factive.

\begin{definition}
For every $e\in\N$, let $\Sigma_e$ be the set of axioms of a knowing machine, along with
Knowledge of having G\"{o}del number $e$, the schema $K(K(\phi)\leftrightarrow\ulcorner\phi\urcorner\in W_e)$.
\end{definition}

\begin{theorem}
\label{theorem7}
There is an $e\in\N$ such that there is a knowing machine which satisfies $\Sigma_e$.
In words: there is a knowing machine with Knowledge of having G\"{o}del number $e$.
\end{theorem}

\begin{proof}
Let $\Sigma_e^{'}$ consist of the axioms of $PA$, along with
all the axioms of $\Sigma_e$ \emph{except} for Factivity, along with the schema
$K(\phi)\leftrightarrow \ulcorner\phi\urcorner\in W_e$ ($\phi$ ranging over sentences).

For $e\in\N$,
by a \emph{coded consequence of $\Sigma_e^{'}$}, I mean the G\"{o}del number
$\ulcorner\phi\urcorner$ of a formula $\phi$ such that $\Sigma_e^{'}\models\phi$.
Given $e\in\N$, we can effectively write a program to enumerate the coded consequences of $\Sigma_e^{'}$.
By the Church-Turing Thesis, there is a total computable function $f:\N\to\N$
such that for every $e\in\N$,
$W_{f(e)}$ is the set of coded consequences of $\Sigma_e^{'}$.
By Kleene's Recursion Theorem, there is an $e\in\N$ such that $W_{f(e)}=W_e$.
Thus, $W_e$ is the set of coded consequences of $\Sigma_e^{'}$.
Fix this $e$ hereafter.

Let $\mathscr{M}$ be the following $\mathscr{L}(K)$-structure.
The universe of $\mathscr{M}$ is $\N$, and symbols of PA are interpreted as usual.
Predicate symbols are interpreted by $\mathscr{M}\models K(\phi)$
iff $\Sigma_e^{'}\models\phi$.
I will show $\mathscr{M}\models\Sigma_e$, proving the theorem.

\begin{itemize}
\item \emph{(Knowledge of Tautology)}  If $\phi$ is a tautology, then $\Sigma_e^{'}\models\phi$,
so $\mathscr{M}\models K(\phi)$.
\item \emph{(Knowledge Modus Ponens)} If $\mathscr{M}\models K(\phi\rightarrow\psi)$ and $\mathscr{M}\models K(\phi)$
then $\Sigma_e^{'}\models\{\phi\rightarrow\psi,\phi\}$, thus $\Sigma_e^{'}\models\psi$, so $\mathscr{M}\models K(\psi)$.
\item \emph{(PA)} With universe $\N$ and interpreting symbols of $PA$ as usual, $\mathscr{M}\models PA$.
\item \emph{(Closure, Knowledge of PA, Knowledge of having G\"{o}del number $e$)}
If $\phi$ is an axiom of $PA$, an instance of $K(\psi)\leftrightarrow\ulcorner\psi\urcorner\in W_e$, or if $K(\phi)$
iis an instance of Closure, then $\phi\in \Sigma_e^{'}$ by construction, so $\Sigma_e^{'}\models \phi$ and
$\mathscr{M}\models K(\phi)$.
\item \emph{(Having G\"{o}del number $e$)}
$\mathscr{M}\models K(\phi)$ iff $\Sigma_e^{'}\models \phi$, iff $\ulcorner\phi\urcorner$
is a coded consequence of $\Sigma_e^{'}$, iff $\ulcorner\phi\urcorner\in W_{f(e)}=W_e$,
which holds iff $\mathscr{M}\models \ulcorner\phi\urcorner\in W_e$ (since $\mathscr{M}$ has universe $\N$ and
interprets symbols of $PA$ as usual).
\item \emph{(Factivity)}
By the previous claims, $\mathscr{M}\models \Sigma_e^{'}$.
Assume $\mathscr{M}\models K(\phi)$.
By definition, $\Sigma_e^{'}\models \phi$.
Since $\mathscr{M}\models\Sigma_e^{'}$, we have $\mathscr{M}\models\phi$.
\end{itemize}
\end{proof}

\end{document}